%% file: main.tex
\documentclass[11pt]{amsart}   
\usepackage{amsmath, amssymb}  
\usepackage{amsthm} 
\usepackage{paralist}

\usepackage{verbatim}
\usepackage{amsfonts}
\usepackage{amssymb}
\usepackage{mathrsfs}
\usepackage[dvips]{graphicx}
\usepackage{pstricks} 
\usepackage{pst-node}

\usepackage{setspace}
\usepackage[inner=2.4cm,outer=4cm,top=2.4cm,bottom=2.4cm,marginparwidth=90pt]{geometry}   
\usepackage{times}

\usepackage[notref,notcite,final]{showkeys}  
  
\numberwithin{equation}{section}

\usepackage{framed}  

\usepackage[dvips]{graphicx}

\newcommand{\bigR}{{\mathbb R}}

\newtheorem{theorem}{Theorem}[section]  
\newtheorem{lemma}[theorem]{Lemma}  
\newtheorem{proposition}[theorem]{Proposition}  
\newtheorem{corollary}[theorem]{Corollary}  
\newtheorem{remark}[theorem]{Remark}   
\newtheorem{definition}[theorem]{Definition}

\allowdisplaybreaks[1]  

\begin{document}  
  
\title{Singularities of axially symmetric volume preserving mean curvature flow}  
\author{Maria Athanassenas}  
\address{Maria Athanassenas, Defence Science and Technology Group, Eveleigh, NSW 2015 \&   School of Mathematical Sciences,  
Monash University,  
Vic 3800  
Australia}  
\email{maria.athanassenas@dst.defence.gov.au} 
  
\author{Sevvandi Kandanaarachchi}  
\address{Sevvandi Kandanaarachchi, Department of Econometrics \& Business Statistics,  Level 8, 20 Chancellors Walk, Clayton Campus,  
Monash University,  
Vic 3800  
Australia}  
\email{sevvandi.kandanaarachchi@monash.edu} 

 \onehalfspacing  
  
\subjclass[2010]{53C44, 35K93} 

\begin{abstract}
We investigate the formation of singularities for surfaces evolving by volume preserving mean curvature flow. For axially symmetric flows - surfaces of revolution - in $\mathbb{R}^{3}$ with Neumann boundary conditions, we prove that the first developing singularity is of Type I. The result is obtained without any additional curvature assumptions being imposed, while axial symmetry and boundary conditions are justifiable given the volume constraint. Additional results and ingredients towards the main proof include a non-cylindrical parabolic maximum principle, and a series of estimates on geometric quantities involving gradient, curvature terms and derivatives thereof. These hold in arbitrary dimensions.   
\end{abstract} 

\maketitle

\section{Introduction}
\noindent
A hypersurface evolves by mean curvature flow if at each point it moves in the direction of its unit normal with speed given by its mean curvature. Assume $M^n$ to be a $n$-dimensional manifold and consider a one-parameter family of smooth immersions $\mathbf{x}_t: M^n \rightarrow \mathbb{R}^{n+1}$. The hypersurfaces $M_t =\mathbf{x}_t \left( M^n \right)$ evolving by mean curvature flow is equivalent to $\mathbf{x}_t = \mathbf{x}(\cdot, t)$ satisfying
\begin{equation}\label{eq:int_2}
\frac{d}{dt} \mathbf{x}(l, t)=-H(l,t)\nu(l,t), \hspace{5 mm} l\in M^n, t>0 \, .
\end{equation}

\noindent
By $\nu(l,t)$ we denote a smooth choice of unit normal of $M_t$ at $\mathbf{x}(l,t)$ (outer normal in case of compact surfaces without boundary), and by $H(l,t)$ the mean curvature with respect to this normal. Surface area is known to decrease under (\ref{eq:int_2}) and, provided the flow converges, the limit is a minimal surface.\\

\noindent
Here we are interested in the evolution of compact surfaces $M_t$ assumed to enclose a prescribed volume $V$. The evolution equation changes by introducing a forcing term as follows: 
\begin{equation}\label{eq:int_1}
\frac{d}{dt} \mathbf{x}(l,t) = -\left( H(l,t) -h(t)\right)\nu(l,t), 
\hspace{5 mm} l\in M^n, t>0,
\end{equation}
\noindent
where $h(t)$ is the average of the mean curvature,
\[h(t) =\frac{\int_{M_t} H dg_t}{\int_{M_t} dg_t}, \]
and $g_t$ denotes the metric on $M_t$. This flow is known to decrease the surface area while the enclosed volume remains constant. A limit surface in this case would have constant mean curvature and be a solution of the isoperimetric problem.\\

\noindent
In this paper, we are interested in the formation of singularities for surfaces evolving by (\ref{eq:int_1}). \\

Extensive research has been undertaken in mean curvature flow, including on long-term geometric behaviour of solutions and the formation of singularities. The selection of references here is mainly guided by the techniques they introduce that are of relevance to our paper. Huisken \cite{GH1} proves  that uniformly convex, compact  surfaces become asymptotically spherical under mean curvature flow. Grayson \cite{GR1} proves that smooth embedded curves in the plane shrink to a point when evolving by curvature flow, becoming spherical in the limit. Ecker and Huisken \cite{EH89} prove that entire graphs of linear growth over $\mathbb{R}^n$ ``flatten out" with time when evolving by mean curvature.  Formation of singularities for (\ref{eq:int_2}) in the non-convex case is considered by Huisken \cite{GH2}, Grayson \cite{GRMA1}, Dziuk and Kawohl \cite{DGKB1},  Altschuler, Angenent and Giga \cite{AAG}, Huisken and Sinestrari  (\cite{HS1},\cite{HS3}).\\


The challenge in the volume preserving mean curvature is the global aspect introduced to equation (\ref{eq:int_1}) by $h$, rendering the use of standard local techniques either impossible or very complicated. In the case of a compact, uniformly convex initial hypersurface $M_0$ without boundary, Huisken \cite{GH3} proves long-time existence for (\ref{eq:int_1}) and convergence to a sphere.  The first author \cite{MA1} proves that an axially symmetric \footnote{By 'axially symmetric' we mean a surface of revolution generated through rotation of a curve. This is for consistency with our previous publication, while \cite{MA1} uses the expression 'rotationally symmetric'.} hypersurface in $D$, which encloses a sufficiently large volume and has Neumann boundary data, converges to a cylinder. While convexity is crucial for  \cite{GH3}, the axial symmetry assumption allows for geometric arguments to be used in \cite{MA1} to overcome difficulties related to the global aspect of $h$.  In \cite{MA2} she proves that thin necks of axially symmetric volume preserving mean curvature flow pinch-off in finite time, the singular set is discrete and finite along the axis of rotation, and that Type I singularities are self-similar and asymptotically cylindrical.  Escher and Simonett \cite{ES98} prove that if the hypersurface is a graph over a sphere with bounds on its height function, then it converges to the sphere under (\ref{eq:int_1}). Hartley combines geometric diffusion techniques with harmonic analysis approaches to show that hypersurfaces 'near' spheres converge to spheres \cite{DH2013}, and hypersurfaces 'near' cylinders converge to either cylinders or, surprisingly, in higher dimensions to half-period unduloids \cite{DH2016}. Cabezas-Rivas and Miquel \cite{CRM09} study the volume flow in hyperbolic space. \\

\noindent
\textbf{Results}\\

\noindent
{\it{Assumptions:}} In this paper we study (\ref{eq:int_1}) and, except for the volume constraint, we have a free boundary. A convexity assumption would not be natural. Instead, we assume axial symmetry and that the surface meets $\partial D$ orthogonally. This is motivated by the fact that stationary solutions to the associated Euler Lagrange equation of an energy minimising liquid bridge contained in $D$ are axially symmetric and satisfy a Neumann boundary condition.

In particular, we are interested in the formation of singularities for surfaces evolving by (\ref{eq:int_1}). We assume the initial compact $n$-dimensional hypersurface $M_0$  to be smoothly embedded in the domain $D = \left\{ \mathbf{x} \in \mathbb{R}^{n+1}, a \leq x_1 \leq b  \right\}$, $a,b >0 $, with boundary $\emptyset \neq \partial M_0 \subset \partial D$. \\

We study the first singularity that develops under this flow (see \cite{MA2} for conditions under which singularities can develop), and prove that it is of Type I: 

\begin{theorem}\label{Theorem_TypeI}
Let $T>0$ denote the time of the first singularity. Then, for  a $2$-dimensional hypersurface $M_0$, satisfying the above assumptions and evolving under \ref{eq:int_1}, the norm of the second fundamental form $|A|$ satisfies
$$ \max_{M_t} |A|^2 \leq  \frac{C}{T-t} \, , $$
for all $t<T$, and where $C\,$ is a constant.
\end{theorem}

\noindent
Our results complement \cite{MA1}, \cite{MA2} and \cite{AthKan1}. While all our height, gradient, curvature and derivatives of the curvature  estimates (including in Section \ref{Section_Height_Gradient_And_Curvature_Estimates}) are valid for flows in arbitrary dimensions, the final Section \ref{Section_Singularity} makes use of results in \cite{GH2} which is restricted to $2$-dimensional surfaces in $\mathbb{R}^{3}$. We also use the explicit parametrisation of a catenoid in three dimensions in the rescaling argument in that final section.\\


The paper is organized as follows:\\
\indent In Section \ref{Section_Notation} we introduce notations and definitions, we present the evolution equations for various geometric quantities and we introduce the different regions of the surface used throughout the paper. These different regions are determined by bounds on curvature terms or combinations thereof, and can be studied individually in regards to the formation of the singularity.\\

In Section \ref{Section_Max_Princ} we discuss parabolic maximum principles for non-cylindrical domains. We extend Ecker's (\cite{KEB1}, Proposition $3.1$) and Lumer's \cite{GL872} versions of the maximum principle to our setting, where it is subsequently used to specific regions of the evolving hypersurfaces determined by conditions on the mean curvature. This means that the base domain varies with time and we need to consider boundary data for theses changing domains as well. The application in our setting does not allow for the luxury of previous work where the surface could be reflected across the boundary and be considered as periodic. The proof is presented in an Appendix to facilitate the flow of the arguments for the reader.\\



In Section  \ref{Section_Height_Gradient_And_Curvature_Estimates} we prove height, gradient and curvature estimates. We prove that the mean curvature is bounded from below on the entire hypersurface. We prove that that the second fundamental form $|A|$ is bounded for subregions of the surface where the radius is bounded from below, so that singularities can only occur along the axis of rotation. Results in this section hold in any dimension. \\




In Section \ref{Section_Singularity} we prove Theorem \ref{Theorem_TypeI} by studying the different cases in which a singularity can develop. In addition to appropriate application of our previously obtained estimates on geometric quantities, a main ingredient here is a rescaling argument similar to that used in \cite{HS2} adapted to our setting, rescaling from points on the axis of rotation. Parts of this section are based on results of \cite{GH2} that only work in $\mathbb{R}^3$, as well as the explicit parametrisation of catenoids in $\mathbb{R}^3$. \\

\noindent
{\small Acknowledgement. The authors thank Prof. Gerhard Huisken and Dr Ben Andrews for helpful discussions; and for supporting visits to the Max Planck Institute for Gravitational Physics, Germany, and the Australian National University respectively. The authors also thank the anonymous reviewers for their suggestions and comments.}

\section{Notation, evolution equations and definitions}\label{Section_Notation}
\noindent
\subsection{Notations}
We follow Huisken's \cite{GH2} and Athanassenas' \cite{MA1} notation in describing the $n$-dimensional axially symmetric hypersurface. Let $\rho_0:[a,b] \rightarrow  \mathbb{R} $ be a smooth, positive function on the bounded interval $[a,b]$ with $\rho'_0(a) = \rho'_0(b)=0 $. Consider the $n$-dimensional hypersurface $M_0$ in $\mathbb{R}^{n+1} $ generated by rotating the graph of $\rho_0$ about the $x_1$-axis. We evolve $M_0$ along its mean curvature vector while keeping its enclosed volume constant and subject to Neumann boundary conditions at $x_1 = a$ and $x_1 = b$. By definition the evolution preserves axial symmetry. The position vector $\mathbf{x}$ of the hypersurface satisfies the evolution
equation
\begin{align}\label{notEq:1.3}
\frac{d}{dt}\mathbf{x} & = -(H-h)\nu = \mathbf{H} +h\nu  \,  ,  \notag \\
  & =   \Delta{\mathbf{x}}+h\nu
\end{align}

\noindent 
where $\mathbf{H}$ is the mean curvature vector, and since $\Delta \mathbf{x} = \mathbf{H}$, where $\Delta$ denotes the Laplacian on the surface.

Let $\mathbf{i}_1, \ldots, \mathbf{i}_{n+1}$ be the standard basis of $\mathbb{R}^{n+1}$, corresponding to $x_1, \ldots x_{n+1}$ axes, and $\tau_1(t), \ldots, \tau_n(t)$ be a local orthonormal frame on $M_t$ such that 
\[ \left\langle\tau_l(t), \mathbf{i}_1\right\rangle = 0, \hspace{2mm} \text{for} \hspace{2mm} l = 2, \ldots, n\, , \hspace{5mm} \text{and} \hspace{5mm} \left\langle\tau_1(t), \mathbf{i}_1\right\rangle > 0  \, .\]

\noindent
Let $ \omega= \frac{\hat{\mathbf{x}}}{|\hat{\mathbf{x}}|} \in \mathbb{R}^{n+1}$ 
denote the unit outward normal to the cylinder intersecting $M_t$ at the point $\mathbf{x}(l,t)\, ,$ where $\hat{\mathbf{x}}= \mathbf{x} - \left\langle \mathbf{x}, \mathbf{i}_1 \right\rangle \mathbf{i}_1$.  Let 

$$y = \langle \mathbf{x}, \omega \rangle \hspace{3 mm} \text{and}  \hspace{3mm} v=\left\langle \omega, \nu \right\rangle ^{-1}\, .$$ 
We call $y$ the { \it height function} and $v$ the {\it gradient function}. We note that $\rho(x_1, t)$ is the radius function such that $\rho: [a, b] \times [0, T) \rightarrow \bigR$, whereas $y(l,t)$ is the height function and $y: M^n \times [0, T) \rightarrow \bigR $.
We note that $v$ is a geometric quantity, related to the inclination angle; in particular $v$ corresponds to $\sqrt{1+ \rho'^2}$ in the axially symmetric setting. The quantity $v$  has facilitated results such as gradient estimates in graphical situations (see for example \cite{EH89,Eck97Mink} ).

\noindent
We introduce the quantities (see also \cite{GH2} )
\begin{equation}\label{notEq:1.1}
 p =  \left\langle\tau_1, \mathbf{i}_1 \right\rangle y^{-1},  \hspace{10 mm} q= \left\langle\nu, \mathbf{i}_1 \right\rangle y^{-1}, 
\end{equation}

\noindent
so that
\begin{equation}\label{notEq:1.2}
p^2 + q^2 = y^{-2} \, .
\end{equation}

\noindent
The second fundamental form has $n-1$ eigenvalues equal to $p = \frac{1}{\rho \sqrt{1+ \rho'^2}}$ and one eigenvalue equal to
\[ k= \left\langle \overline{\nabla}_{1} \nu, \tau_1 \right\rangle = \frac{-\rho''}{(1+\rho'^2)^{3/2}}. \]

\noindent
There are cases where singularities develop in the axially symmetric setting (see \cite{MA2}). Here, we assume that a singularity develops for the first time at $t=T < \infty$. 

\subsection{Evolution equations}
\noindent
We have the following evolution equations: 
\begin{lemma}\label{Lemma_evolutionEqs}
We have the following evolution equations:
\end{lemma}
\begin{itemize}
\item[(i)] $\frac{d}{dt}\left\langle \mathbf{x}, i_1 \right\rangle = 
\Delta \left\langle \mathbf{x}, i_1 \right\rangle + hqy \, ; $
\item[(ii)] $\frac{d}{dt} y = -(H-h)py =\Delta y - \frac{n-1}{y} + hpy \, ; $ \label{eq:alt_y}
\item[(iii)] $\frac{d}{dt} q = \Delta q + |A|^2q + q((n-1)p^2 + (n-3)q^2 - 2kp) - hpq \, ;  $
\item[(iv)] $ \frac{d}{dt} p = \Delta p + |A|^2p + 2q^2(k-p) - hp^2 \, ;  $
\item[(v)] $ \frac{d}{dt} k = \Delta k + |A|^2k -2(n-1)q^2(k-p) - hk^2 \, ;  $
\item[(vi)]$ \frac{d}{dt} H = \Delta H + (H -h) |A|^2 \, ; $
\item[(vii)] $ \frac{d}{dt}  |A|^2= \Delta|A|^2 -2|\nabla A|^2 + 2 |A|^4 - 2h\mathscr{C}\, ;  $
\item[(viii)]$ \frac{d}{dt} v = \Delta v  -|A|^2v + (n-1)\frac{v}{y^2} - \frac{2}{v}|\nabla v|^2\, ;  $
\item[(ix)] $\frac{d}{dt} \rho = \frac{\rho''}{1+ \rho'^2} - \frac{n-1}{\rho} + h \sqrt{1+ \rho'^2}  \, ;$
\end{itemize}
where $\mathscr{C} = g^{ij}g^{kl}g^{mn}h_{ik}h_{lm}h_{nj}\, .$

\begin{proof}
The evolution equations are either proved in \cite{GH3}, \cite{MA1}, \cite{MA2}, and \cite{AthKan1} or similar to those in \cite{GH2}. Equations (ii), (vi), (vii) and (viii) are derived in (\cite{MA1}, Lemma $3$), where for (ii) we use $ \left\langle \nu, \omega \right\rangle = \frac{1}{v} = py $. Equations (iv) and (v) are derived in \cite{AthKan1} and  (ix) in  \cite{MA2}. Equation (i) follows from (\ref{eq:int_1}) using (\ref{notEq:1.3}). Equations (iii) is as in \cite{GH2} (Lemma $5.1$) adjusted to the volume constraint.

\subsection{Bounds on $h$}
\noindent
We state \cite{MA2} Proposition $1.4$ here. 
\begin{proposition}\label{Maria's Prop 1.4}{\bf (Athanassenas).} Assume $\{M_t\}$ to be a family of smooth, rotationally
symmetric surfaces, solving \eqref{eq:int_1} for $t \in [0, T) \, . $ Then the mean value $h$ of the mean curvature satisfies
$$0 < c_2 \leq  h \leq c_3 \, , $$
with $c_2$ and $c_3$ constants depending on the initial hypersurface $M_0$.
\end{proposition}

This is an important result that will be used repeatedly in our paper. 

\subsection{Different regions of the volume flow surface }
\noindent
Depending on the situation, we are interested in different parts of the hypersurface; therefore we subdivide as follows :

\subsubsection{The regions $\breve\Omega_t \, , \hat\Omega_t$ and $\breve\Omega'_t$ }
Let 
$$ \breve{\Omega}_t =\left\{\mathbf{x}(l,t) \in M_t : H(l,t) \leq \frac{c_2}{2}\right\} \,  \hspace{3mm} \text{and} \hspace{3mm} \breve\Omega = \bigcup_{t < T}  \breve{\Omega}_t \, , $$
$$ \hat{\Omega}_t = \left\{\mathbf{x}(l,t) \in M_t : H(l,t) > \frac{c_2}{2}\right\} \,  \hspace{3mm} \text{and} \hspace{3mm} \hat\Omega = \bigcup_{t < T} \hat{\Omega}_t \, , $$
such that $M_t = \breve{\Omega}_t \cup \hat{\Omega}_t$. We also define 
$$  \breve\Omega'_t = \left \{\mathbf{x}(l,t) \in M_t : H(l,t) \leq c_2 - \delta , \delta > 0 \right \}\,  \hspace{3mm} \text{and} \hspace{3mm}    \breve\Omega' =  \bigcup_{t < T} \breve\Omega'_t \, ,$$ 
\noindent
which will be used occasionally. 

\section{Maximum Principles}\label{Section_Max_Princ}
\noindent
We are interested in maximum principles for non-cylindrical domains in order to be able to work on sub-regions of the hypersurface. This section is an extension of Ecker's (\cite{KEB1}, Proposition $3.1$) and Lumer's \cite{GL872} version of maximum principles to our setting. In \cite{GL872} the maximum principles are proved in an operator theoretic setting, which has been adapted to the manifold setting here. \\

Let $\Omega = M^n$. Let $V \subset \Omega \times (0, T)$ be an open non-cylindrical domain. Let $\Omega_t = \Omega \times \{ t \} \, ,$ and for $t \neq 0 $ let  $V_t = \Omega_t \cap V $, the cross sections of $V$ for constant $t$. Let $\overline{V}$ denote the closure of $V$ and $V_0 = \Omega_0 \cap \overline{V} \, .$  The boundary of $V$ \, is $\partial V = \overline{V} \backslash V \, .$ The parabolic boundary is $\Gamma_V = \partial V \backslash \Omega_T  \, .$ To describe the horizontal parts of the boundary of $V$ in the space-time diagram, we define the following: let $Z_t$ be the largest subset of $\Omega_t \cap \partial V $ that is open in $\partial V$ and can be reached from ``below'' (with $t$ the vertical axis) in $V$ . Let $Z_V = \bigcup_{0 < t < T } Z_t $ and $\delta_V =  \Gamma_V \backslash Z_V  \, .$

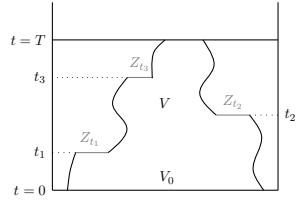
\begin{figure}[h!]
   \centering
   \scalebox{0.5}{\input{nonCylindrical1}}
		\caption{The non-cylindrical domain $V$, with $\delta_V$ indicated by a darker line for $t<T$.}
		\label{fig_nonCylindrical1}
\end{figure}

\begin{proposition}\label{Prop_Max_Principles}{\bf(Non-Cylindrical Maximum Principle )}
Let $\left(M_t \right)_{t \in (0, T)}$ be a solution of the volume preserving mean curvature flow \eqref{eq:int_1} consisting of hypersurfaces $M_t=\mathbf{x}_t(\Omega)$, where $ \mathbf{x}_t = \mathbf{x}( \cdot, t) : \Omega \times [0, T) \rightarrow \bigR^{n+1}$ and $\Omega$ is compact. Suppose $f \in  C^{2,1}(V) \cap C(\overline{V}) $  satisfies an inequality of the form
$$ \left( \frac{d}{d t} - \Delta \right) f \leq \langle a , \nabla f \rangle \, ,
$$ 
where the Laplacian $\Delta$ and the gradient $\nabla$ are computed on the manifold $M_t$. For the vector field $a: V \rightarrow \bigR^{n+1}$ we only require that it is continuous in a neighbourhood of all maximum points of $f$ . Then
$$ \sup_V f \leq  \sup_{\Gamma_V} f \, ,
$$
for all $t \in [0, T)$.  \\
Assuming $f$ to have a positive supremum in $V$ then
$$ \sup_V f \leq \sup_{\delta_V} f \, ,
$$
for all $t \in [0, T) \, .$
\end{proposition}
For the convenience of the reader, the proof is included in the Appendix.


\section{Height, Gradient and Curvature Estimates}\label{Section_Height_Gradient_And_Curvature_Estimates}

In this section we prove radius estimates from below in $\breve\Omega$; various curvature estimates, including for the ratio $\frac{|k|}{p}$ of the principal curvatures; and, for the norm $|A|$ of the second fundamental form on any subregion of the evolving hypersurface away from the axis of rotation.

\subsection{Height estimates}
\noindent
The first author proves in ( \cite{MA1}, 2A Remark (iii)) that the height $y$ satisfies
$$ y \leq R \, , $$
\noindent
for some $R>0$ determined by the initial hypersurface $M_0$. \\
\noindent
We will show that the height function $y$ has a lower bound in the region $\breve\Omega$. 

\begin{lemma}\label{Lemma_yLowerBound}There exist constants $c , c' >0 $  such that $\inf_{ \breve\Omega} y = \inf_{\Gamma_{ \breve{\Omega}}} y \geq c$ and  $ \inf_{ \Gamma_{ \hat{\Omega}}} y \geq c' \, ,$ where $\Gamma_{ \breve{\Omega}}$ and $\Gamma_{ \hat{\Omega}}$ denote the parabolic boundary of $\breve{\Omega}$ and $\hat{\Omega}$ (see figure \ref{figure_yLowerBound}) respectively. 
\end{lemma}

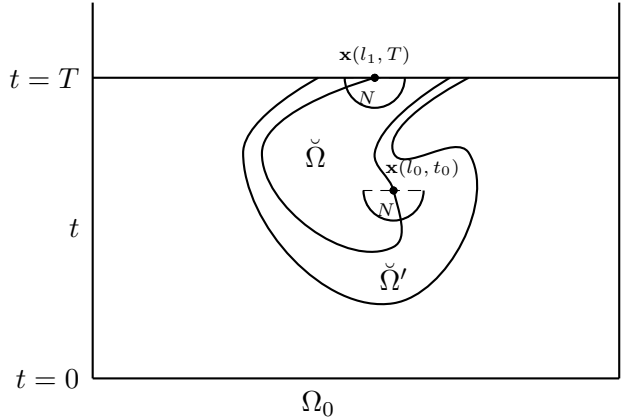
\begin{figure}[h]
   \centering
   \input{NegativeHBdd}
		\caption{Space time schematic diagram for $\breve{\Omega}$ and $ \breve{\Omega}'$}
		\label{figure_yLowerBound}
\end{figure}
\begin{proof}
For this proof we work with $\breve\Omega$ and  a set containing it such that $H < c_2$. In particular, we can choose $\delta$ such that $\frac{c_2}{2} < c_2 - \delta$, and work with $ \breve\Omega \subset  \breve\Omega'$ . As $\frac{dy}{dt} = -(H-h)py > 0 $ in $\breve\Omega\, ,$ the height increases in this region. Therefore
\[ \inf_{ \breve\Omega } y = \inf_{ \Gamma_{ \breve\Omega}} y \, . \]
\noindent
 We claim that $\inf _{\Gamma_{ \breve\Omega } } y \neq 0 \, .$ To prove this suppose $\inf_{\Gamma_{ \breve\Omega } } y= 0 \, ,$ at a point $\mathbf{x}(l_0, t_0) \in \overline{\breve\Omega} \, $ (see figure \ref{figure_yLowerBound}), where $\overline{\breve\Omega}$ is the closure of $\breve\Omega$ and $t_0$ may equal to $T$.  If the height is zero at the point $\mathbf{x}(l_0, t_0) \in \overline{\breve\Omega} $, then the height has to decrease near the point just before $t_0$. That means there exists a neighbourhood $N$ of $\mathbf{x}(l_0, t_0)\, ,$  such that $N$ is ``past" in time, $ t < t_0$, and $N \subset \breve{\Omega}'\, ,$  and  $\frac{dy}{dt}{\big\vert}_{N} < 0 \, .$  But this is not possible, since $\frac{dy}{dt}{\big\vert}_{\breve{\Omega}'}  > 0 \, .$  
 Therefore there exists a constant $c$ such that, on the parabolic boundary of $\breve\Omega \, ,\inf y \geq c > 0 \, .$\\
\noindent
When we consider $M_t$ as a periodic hypersurface, we have $\partial \breve\Omega_t = \partial \hat\Omega_t \, ,$ that is $\Gamma_{ \breve\Omega} \backslash \breve\Omega_0 =\Gamma_{ \hat\Omega} \backslash \hat\Omega_0  $. As $\inf_{M_0} y \neq 0$ we have the desired result.
\end{proof}

\begin{remark} \label{Rem_Blowup} Similarly, as in $\breve\Omega$, a lower height bound can be obtained in $\breve\Omega'$ for any $\delta >0$.
\end{remark}

\subsection{A Gradient estimate}

The following Lemma gives us scaling control over the gradient compared to the radius when approaching a singularity on the axis of rotation. 

\begin{lemma} \label{Lemma_UV} There exists a constant $c_4$ depending only on the initial hypersurface, such that $ vy < c_4$, independent of time. 
\end{lemma}
\begin{proof}
We calculate from Lemma \ref{Lemma_evolutionEqs}
\[ \frac{d}{dt} (yv -c_3 t) = \Delta(yv) -\frac{2}{v}\left\langle \nabla v \, , \nabla (yv) \right\rangle - yv|A|^2 + h-c_3 \, .\]
As $h \leq c_3$ we get by the parabolic maximum principle 
\begin{align*}
yv-c_3t &\leq \max_{M_0} yv \, , \\
yv &\leq \max_{M_0} yv +c_3T =: c_4 \, .
\end{align*}
\end{proof}

\subsection{Curvature estimates}

The next two propositions allow us to control the ratio of the principal curvatures on all of $M_t$ independent of time.

\begin{proposition} \label{Prop_KonP}There is a constant $c_1$ depending only on the initial hypersurface, such that  $\frac{k}{p} < c_1$, independent of time.
\end{proposition}
\begin{proof}
Similar to equation $(19)$ of \cite{GH2} we calculate from Lemma \ref{Lemma_evolutionEqs}
\[ \frac{d}{dt} \left( \frac{k}{p} \right) = \Delta \frac{k}{p} +
	\frac{2}{p} \left\langle \nabla p \, , \nabla \left(\frac{k}{p}\right) \right\rangle +
	2\frac{q^2}{p^2}\left(p - k\right)\left((n-1)p + k\right) +
	\frac{hk}{p}\left(p-k\right) \, .\]
If $\frac{k}{p} \geq 1$ then $\left(p-k\right)<0$. By the parabolic maximum principle we obtain
\begin{equation}
\frac{k}{p} \leq \max \left(1, \max_{M_0} \frac{k}{p} \right) =: c_1 \, .
\end{equation}
\end{proof}

\begin{proposition} \label{Prop_KonP2} At points  $\mathbf{x}(l,t)$ of $M_t$ where $H\geq 0$ we have $\frac{|k|}{p} \leq \max(c_1, n-1)$. 
\end{proposition} 
\begin{proof}
In a region or at any given point where $H$ is positive, if $k$ is positive as well  we have by Proposition \ref{Prop_KonP} $\frac{|k|}{p} = \frac{k}{p} \leq c_1$. If $k$ is negative, then  
\begin{align}\label{evEq:1.8}
 k + (n-1)p \geq 0 \, ,\notag \\
 -|k| + (n-1)p \geq 0 \, , \notag \\
 \frac{|k|}{p} \leq (n-1)\, .
\end{align} 
\end{proof}

Now we proceed to show that singularities cannot develop away from the axis of rotation, and that $|A|$ is bounded in regions where $ y \geq \epsilon >0 $.  

\begin{proposition}\label{Prop_AsquaredBddAwayFromAxis} For given $\epsilon > 0 $, let $S_t \subset M_t$ and $S = \bigcup_{t < T} S_t\, $ be a region such that $y \vert_S \geq \epsilon > 0\, $  and $H\vert_{\partial S_t} \geq 0 $ for all $t < T \, . $ Then the norm of the second fundamental form $|A|$ is bounded in $S$. 

\end{proposition}
\begin{proof}
We proceed as in (\cite{EH91}, proof of Theorem $3.1$), (\cite{MA1}, Proposition $5$) and (\cite{AthKan1}, Proposition 6.2) and calculate the evolution equation for the product $g = |A|^2 \varphi(v^2)$, where $\varphi(r) = \frac{r}{\lambda - \mu r}$, with some constants $\lambda, \mu >0 $ to be chosen later and $v = \langle \nu, \omega \rangle ^{-1}$. From the evolution equation of $g$ we find the inequality
\[\left ( \frac{d}{dt}- \Delta \right) g \leq -2\mu g^2 -2\lambda \varphi v^{-3} \left\langle \nabla v \, , \nabla g \right\rangle - \frac{2\lambda \mu}{(\lambda-\mu v^2)^2}|\nabla v|^2 g -2hC \varphi(v^2) 
+ \frac{2(n-1)}{y^2}v^2 \varphi'|A|^2  \, .\]
Similar to (\cite{AthKan1}, Proposition 6.2) we obtain

\[ g \leq \max \left( \max_{\Gamma _S} g, \, C \right)  \, ,  \] 
where $\Gamma_S$ denotes the parabolic boundary of $S$, which may be non-cylindrical.  Therefore from the non-cylindrical maximum principle (Proposition \ref{Prop_Max_Principles}) 
\begin{equation}\label{eq_Ev_Eq_1.43}
 |A|^2\varphi(v^2)  \leq \max \left( \max_{S_0} |A|^2\varphi(v^2), \hspace{2mm}  \max_{\begin{subarray}{l} {\partial S_t}\\ t<T \end{subarray}} |A|^2\varphi(v^2), \, C  \right) \, . 
\end{equation}
\noindent
As $H\vert_{\partial S_t} \geq 0 $ for all $t<T\, ,$ we have $\frac{|k|}{p} {\big \vert}_{\partial S_t} < \max(c_1, n-1) =: c^1 $ for all $t<T\, ,$  by Proposition \ref{Prop_KonP2}.  Therefore
\begin{align*}
 |A|^2\vert_{\partial S_t} &= (k^2 + (n-1) p^2 )\vert_{\partial S_t} \leq (n-1 + c^1)p^2 \vert_{\partial S_t} \, , \\
 & \leq (n-1 + c^1) y^{-2}\vert_{\partial S_t} \leq (n-1 + c^1) \epsilon^{-2} \, , 
 \end{align*}
\noindent
for all $t<T$. Note that $\varphi(v^2)>0$ and is bounded from above as long as $v$ is bounded, which holds for any points that are at a distance larger than $\epsilon$ from the axis of rotation (Lemma \ref{Lemma_UV}). Therefore $\max_{\begin{subarray}{l} {\partial S_t}\\ t<T \end{subarray}} |A|^2\varphi(v^2) $ is bounded. Thus $ |A|^2\varphi(v^2) $ is bounded in $S$.  As we chose $\lambda$ to be greater than $ \mu \max v^2 $ and as $v \geq 1 $, we have $\left(\varphi(v^2)\right)^{-1}$ bounded as well and this completes the proof.
\end{proof}

Proposition \ref{Prop_AsquaredBddAwayFromAxis} gives in particular a bound on $|A|$ in $\breve\Omega$, which includes all regions of negative $H$; we summarize this result here:

\begin{corollary}\label{Lemma_HLowerBound}
For $\mathbf{x}$ satisfying \eqref{eq:int_1}, the norm of the second fundamental form $|A|$ is bounded in the region $\breve\Omega$.  
\end{corollary}
\begin{proof}
\noindent
From Lemma \ref{Lemma_yLowerBound} we know that in $\breve\Omega \, , \inf y \geq c > 0 \, .$ On the boundary of $ \breve\Omega_t\, , H = \frac{c_2}{2} >0 \, $ for all $t<T$. Therefore by Proposition \ref{Prop_AsquaredBddAwayFromAxis} there exists a constant $C'$ such that $|A|^2 \vert_{\breve\Omega} \leq C' < \infty \, . $ \\
\end{proof}

\begin{proposition}\label{Lemma_HLowerBound_2}
There exists a constant $C$ independent of time such that  $H(l, t) \geq -C^2$ for all $\mathbf{x}(l, t) \in \ M_t$. 
\end{proposition}
\begin{proof}
\noindent
By definition $\breve\Omega = \bigcup_{t < T}  \left\{\mathbf{x}(l,t) \in M_t : H(l,t) \leq \frac{c_2}{2}\right\}$. $H$ can only be negative for $\mathbf{x}(l, t) \in \breve\Omega$. But by the above result $|A|^2 \vert_{\breve\Omega} \leq C' < \infty \, . $ As $\frac{1}{n}H^2 \leq |A|^2$ we deduce
$$ H \vert_{M_t} \geq - \sqrt{nC'}=: -C^2 $$

\end{proof}

We can now refine Proposition \ref{Prop_AsquaredBddAwayFromAxis} to show that no singularities develop away from the axis of rotation.

\begin{proposition}\label{Prop_No_Singularities_Away_From_The_Axis}
For given $\epsilon >0 $, let $S_t \subset M_t$ and $S = \bigcup_{t < T} S_t\, ,$ such that $y \vert_S \geq \epsilon > 0\, ,$ for all $t < T \, . $ Then the norm of the second fundamental form $|A|$ is bounded in the region $S$.
\end{proposition}
\begin{proof}
The proof is exactly the same as in Proposition \ref{Prop_AsquaredBddAwayFromAxis} up to \eqref{eq_Ev_Eq_1.43}. 

\noindent
Let 
\begin{align*}
 \partial S_t^+ &= \left\{ \mathbf{x}(l,t) \in \partial S_t \colon H(l,t) \geq 0  \right\} \, , \hspace{3mm} \text{and} \\
\ \partial S_t^- &= \left\{ \mathbf{x}(l,t) \in \partial S_t \colon H(l,t) < 0  \right\} \, , 
\end{align*}
\noindent
so that $\partial S_t = \partial S_t^+ \cup \partial S_t^-$. Continuing from \eqref{eq_Ev_Eq_1.43}
$$  |A|^2\varphi(v^2)  \leq \max \left( \max_{S_0} |A|^2\varphi(v^2), \hspace{2mm}  \max_{\begin{subarray}{l} {\partial S_t^+}\\ t<T \end{subarray}} |A|^2\varphi(v^2), \hspace{2mm}  \max_{\begin{subarray}{l} {\partial S_t^-}\\ t<T \end{subarray}} |A|^2\varphi(v^2), \, C \right) \, .  
$$ 

\noindent
Here we look at the term $\max_{\begin{subarray}{l} {\partial S_t^-}\\ t<T \end{subarray}} |A|^2\varphi(v^2)\, ,$ as the other terms are taken care of in Proposition \ref{Prop_AsquaredBddAwayFromAxis}. As $H <0 < \frac{c_2}{2}$ on $\partial S_t^-$,  $\partial S_t^- \subset \breve\Omega_t$. From Corollary \ref{Lemma_HLowerBound},  $\vert A \vert ^2$ is bounded in $\breve\Omega$. As $\partial S_t^- \subset \breve\Omega_t$,  $\vert A \vert ^2$ is bounded on $\partial S_t^-$ for all $t<T$. As $\varphi(v^2)>0$ and is bounded from above at points away from the axis, $\max_{\begin{subarray}{l} {\partial S_t^-}\\ t<T \end{subarray}} |A|^2\varphi(v^2)$ is bounded and as in \cite{AthKan1} we get the desired result. 
\end{proof}

We proceed now to show that the projection of $\hat{\Omega}$ onto the $x_1$-axis is not 'collapsing' to a point. This result is important for the rescaling argument in Section \ref{Section_Singularity}.

\begin{lemma}\label{Lemma_GradH_on the_x1_axis}For the mean curvature $H$ of the evolving hypersurface $M_t$ we have, if $|\nabla H| \leq c$ in a closed region $S \subset \Omega\times [0, T]$, then  ${\big\vert}\frac{d H}{d x_1}{\big\vert} \leq c $ in $S$ as well.
\end{lemma}
\begin{proof}
\noindent
For the magnitude of $\nabla H$ we obtain
$$ |\nabla H |^2 = |\nabla_{\tau_1} H |^2 + \cdots + |\nabla_{\tau_n} H |^2 \, . $$
\noindent
As $M_t$ is an axially symmetric surface, the mean curvature $H$ is constant on the $n-1$ dimensional sphere for a fixed $x_1$ coordinate. Here we let $\mathbf{x}(l, t) =  \mathbf{x}(x_1, \theta_1, \cdots, \theta_{n-1} , t)$, and for $2 \leq i \leq n$ 
$$ \nabla_{\tau_i} H =  \frac{\partial}{\partial \theta_{i-1}}H = 0 \, , $$ 
\noindent
so that
$$ |\nabla H |^2 = |\nabla_{\tau_1} H |^2  = {\bigg\vert}\frac{\partial}{\partial x_1}H {\bigg\vert}^2  \, . $$
\noindent
As $|\nabla H|$ is bounded in $S$, we have the same bound for ${\big\vert}\frac{\partial H}{\partial x_1}{\big\vert} $ as well.
\end{proof}

\begin{figure}[h]
   \centering
   \scalebox{0.75}{\input{TheTwoPaths}}
		\caption{The paths of $H = C_1 $ and $H = C_2$ in $ \breve{\Omega}'$}
\end{figure}
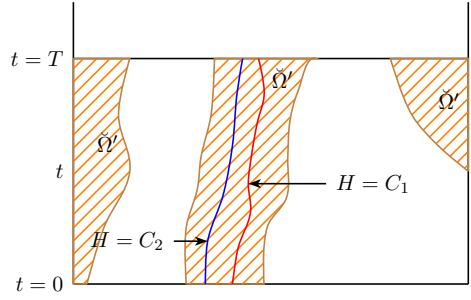

We recall that $\breve\Omega' = \bigcup_{t < T} \breve\Omega'_t = \bigcup_{t < T} \{\mathbf{x}(l,t) : H(l,t) \leq c_2 - \delta : \delta > 0 \}\,  .$  Let us define the following:

\begin{definition}\label{def_6.7}
In a connected component of $ \breve{\Omega}' \, ,$ consider any two paths where $H =C_1$ and $H = C_2 $ such that $0 \leq C_1 < C_2 \leq c_2 - \delta \, , \delta > 0 $. (Recall that $ 0 < c_2 \leq h(t) \leq c_3 $.) Let 
$$ l_1(t) = \{ l  \in M^n : H(l,t) = C_1 \} \, , \hspace{3mm} \text{and} \hspace{3mm}
x_1(l_1(t) , t) = \langle \mathbf{x}(l_1(t), t) , i_1 \rangle \, , $$ 
$$ l_2(t) = \{ l  \in M^n : H(l,t) = C_2 \} \, , \hspace{3mm} \text{and} \hspace{3mm}
x_1(l_2(t) , t) = \langle \mathbf{x}(l_2(t), t) , i_1 \rangle \, , $$ 
$$ \alpha(t) = \min \{x_1(l_1(t) , t),   x_1(l_2(t) , t)  \}   \, , \hspace{3mm} \text{and} \hspace{3mm}
 \beta(t) = \max\{x_1(l_1(t) , t),   x_1(l_2(t) , t)  \} \, . $$
 \noindent
 Here $l_i(t), i=1,2$ is the curve in $M^n\times[0,T)$ that parametrizes $H=C_i$ and $x_1(l_i(t) , t)$ the corresponding $x_1$ coordinate.
\end{definition}

\begin{lemma}\label{Grad_H_Consequences}
With the above notation, there exists a constant $c$ such that $|x_1(l_1(t) , t) -   x_1(l_2(t) , t) | \geq c > 0 $ for all $t \leq T$.
\end{lemma}
\begin{proof}
\noindent
From Remark \ref{Rem_Blowup} we know that $y\mid_{\breve\Omega'} \geq y\mid_{\Gamma_{\breve\Omega'}} \geq \epsilon > 0 $. As the height is always positive in $\breve\Omega'$ we have  $\rho(x_1, t) \in C^{\infty} ( \bigR \times [0, T])$ in that region by (\cite{MA2}, Lemma $2.5$). We note that this holds in $\breve\Omega'$ even at $t=T$, as the height is strictly positive.  Thus, there exists a constant $C$ such that $|\nabla A|\mid_{\breve\Omega'} < C$ for all $t \in [0, T] $  . As $|\nabla H|^2 \leq n |\nabla A|^2$ we have bounds for $|\nabla H|$. From Lemma \ref{Lemma_GradH_on the_x1_axis} we know that there exists a constant $c'$ such that ${\big\vert} \frac{\partial H}{\partial x_1} {\big\vert} \leq c'$ in $\breve\Omega'$, for $t \in [0, T]$. Therefore
$$ {\bigg\vert} \int_{\alpha(t)}^{\beta(t)} \frac{\partial H}{\partial x_1} dx_1  {\bigg\vert} \leq  \int_{\alpha(t)}^{\beta(t)} {\bigg\vert} \frac{\partial H}{\partial x_1} {\bigg\vert}  dx_1\leq    c'\int_{\alpha(t)}^{\beta(t)}  dx_1  \, ,$$
$$ \vert H (\beta(t) , t) - H (\alpha(t), t) \vert \leq c' (\beta(t) - \alpha(t)) \, , $$
$$\frac{\left(C_2 - C_1\right)}{c'} \leq {\big\vert} x_1(l_1(t), t) - x_1(l_2(t), t) {\big\vert} \, ,  \hspace{3mm} \text{for all } \hspace{3mm} t \in [0, T]. $$
\end{proof}

For the next Lemma we recall that 
$$ \hat{\Omega}_t = \left\{\mathbf{x}(l,t) : H(l,t) > \frac{c_2}{2} \right\} \, , \hspace{3mm} \text{and}  \hspace{3mm} \hat{\Omega} = \bigcup_{t < T} \hat{\Omega}_t \, .$$
The following Lemma is important, as the singularity can only develop in $\hat{\Omega}$, a region on the surface that does not 'collapse' to a point.
\begin{lemma}\label{Lemma_width_of_hat_Omega_dashed} There exists a constant $c>0\, ,$ such that the one-dimensional Hausdorff measure of the projection of $\hat{\Omega}$ onto the $x_1$ axis satisfies $\mathscr{H}^1(I(\hat{\Omega})) > c$ for all $t \leq T$.
\end{lemma}
\begin{proof}
To invoke Lemma \ref{Grad_H_Consequences}, we choose  $\delta = \frac{c_2}{20}$, \, $C_1 = \frac{6c_2}{10}$ and $C_2 =\frac{9c_2}{10}$, in accordance with definition \ref{def_6.7}.  Hence $\breve\Omega'_t = \left\{\mathbf{x}(l,t) : H(l,t) \leq \frac{19}{20}c_2 \right \}$. From Lemma \ref{Grad_H_Consequences} we have $|x_1(l_1(t) , t) -   x_1(l_2(t) , t) | \geq c > 0\, , $ for all $t \leq T$. By the definition of $\hat{\Omega}_t$  and choosing $l_i(t),\, i=1,2,$ such that $H(l_i(t), t) = C_i$ we know that $\mathbf{x}(l_i(t), t) \in \hat{\Omega}_t \, , i = 1,2$. As $\mathscr{H}^1(I(\hat{\Omega})) \geq  |x_1(l_1(t) , t) -   x_1(l_2(t) , t) | $ for any $t$, we have the desired result. 
\end{proof}

\section{The Singularity}\label{Section_Singularity}
\noindent
We break up the investigation of the singularity into two cases, depending on the value of $|A|^2/H^2$. From now on all our calculations are done in $\bigR^3$ for two dimensional surfaces as we will use results from \cite{GH2} and the explicit parametrisation of a $2$-dimensional catenoid.

\subsection{The region $S$}
\noindent
Let $S_t \subset \hat{\Omega}_t$  and $S = \bigcup_{t<T} S_t$. For this region $S$ we assume that there exist constants $c_{12} \, , c_{13} > 0 $ such that 
$$ \frac{|A|^2}{H^2} \bigg\vert_S \leq c_{12} \,  \hspace{4mm} \text{and} \hspace{4mm} y\vert_{\Gamma_S} \geq c_{13} \,  .
$$
\noindent
The following Lemma, where we prove a gradient bound in regions of bounded $ \frac{|A|^2}{H^2}$, corresponds to Lemma $5.2$ in \cite{GH2}.
\begin{lemma}\label{Lemma_QonPBdd}
Under the above assumptions, there exists a constant $c_{15}$ such that $ \frac{|q|}{p} \leq c_{15} \, $  in $S_t \,$ for all $t <T \, .$   
\end{lemma}
\begin{proof}
From Lemma \ref{Lemma_evolutionEqs} we compute the evolution equation for $\frac{q}{H}$
\begin{equation*}
\frac{d}{dt} \left( \frac{q}{H} \right ) = 
 	\Delta\left( \frac{q}{H} \right) + \frac{2}{H} \left\langle \nabla H \hspace{1mm} , \nabla \left( \frac{q}{H} \right) \right\rangle + \frac{q}{H} \left( ( p^2 - q^2 - 2kp ) + \frac{h}{H}(k^2-kp) \right) \, . 	
\end{equation*}
We know in $S_t$
$$ |A|^2 \leq c_{12} H^2 \, ,\hspace{6mm} -2kp \leq k^2 + p ^2 = |A|^2 \leq c_{12} H^2 \, ,  $$
$$ 0 < c_2 \leq h \leq c_3 \, , \hspace{6mm} H > \frac{c_2}{2}\, , \hspace{2mm} \text{such that } \hspace{2mm}   \frac{h}{H} \leq \frac{2c_3}{c_2} \, . $$ 
When $\frac{q}{H} > 2 \sqrt{\frac{2c_3c_{12}}{c_2}} $ we obtain
\begin{align*}
 p^2 - q^2 - 2kp  + \frac{h}{H}(k^2-kp) &\leq c_{12} H^2 - 8\frac{c_3}{c_2}c_{12}H^2 + c_{12}H^2 + \frac{2c_3}{c_2} \left( c_{12} H^2 + \frac{1}{2}c_{12}H^2 \right)  \\
 					&=c_{12}H^2 \left( 2 - \frac{5c_3}{c_2} \right) \leq 0 \, ,
\end{align*}
as $\frac{c_3}{c_2} \geq 1 \, .$
Therefore, when $\frac{q}{H} > 2 \sqrt{\frac{2c_3c_{12}}{c_2}}\, , $ from the non-cylindrical maximum principle (Proposition \ref{Prop_Max_Principles} )
$$ \frac{q}{H} \leq \max \left(  \max_{S_0} \frac{q}{H} , \, \hspace{2mm} \max_{\begin{subarray}{l} {\partial S_t}\\ t<T \end{subarray}}  \frac{q}{H}, \, \hspace{2mm} 2 \sqrt{\frac{2c_3c_{12}}{c_2}} \right) \, .$$

We recall that $q=\langle \nu, \mathbf{i}_1 \rangle y^{-1}$. As $\max_{\begin{subarray}{l} {\partial S_t}\\ t<T \end{subarray}} \frac{q}{H} \leq \max_{\begin{subarray}{l} {\partial S_t}\\ t<T \end{subarray}}\frac{2y^{-1}}{c_2}$ and as $y^{-1}\vert_{\begin{subarray}{l} {\partial S_t}\\ t<T \end{subarray}}  \leq \frac{1}{c_{13}}\, ,  $  the right hand side in the above estimate is bounded. \\
\noindent
Similarly when $\frac{q}{H} < -2 \sqrt{\frac{2c_3c_{12}}{c_2}} $ we have
$$ \frac{d}{dt} \left(\frac{q}{H} \right) \geq \Delta \left(\frac{q}{H} \right) + \frac{2}{H} \left\langle \nabla H \hspace{1mm} , \nabla \left( \frac{q}{H} \right) \right\rangle \, . $$

Therefore we have $\frac{|q|}{H} \leq c_{14} \, . $ As $\frac{k}{p} \leq c_1$ we obtain
$$ |q| \leq c_{14} H = c_{14} (p+ k) \leq c_{15} p \, ,$$
as desired.
\end{proof}

\noindent 
\begin{remark}\label{Sing_Domain_Does_Not_Collapse}
(i) Note that $\frac{|q|}{p}$ is a geometric quantity that corresponds to the slope $\vert \rho' \vert$ of the generating curve $\rho$. Therefore, Lemma \ref{Lemma_QonPBdd} gives us a gradient bound in the region $S$. \\ \\
\noindent
(ii) Assuming that the singularity develops in $S$, there is a point on the generating curve that approaches the axis of rotation as $t \to T$. By definition, as $y\vert_{\Gamma_S} \geq c_{13}$, we have $y \vert_{\partial S_t} \geq c_{13}$ for all $t<T$. As the boundary of the domain has a lower height bound, and as the above gradient bound holds in that domain, $S_t$ cannot collapse to a point as $t$ goes to $T$.
\end{remark}

\noindent
The following Proposition corresponds to (\cite{GH2}, Proposition $5.3$) adjusted to the volume constrained case.

\begin{proposition}\label{Prop_TypeI}
If the singularity develops in the region $S$, then there exists a constant $C>0$ such that the second fundamental form satisfies
$$ \max_{S_t} |A|^2 \leq C \frac{1}{T-t} \, , $$
for all $t<T$.
\end{proposition}
\begin{proof}

Using Lemma \ref{Lemma_evolutionEqs} (ii) we have
$$ \frac{d}{dt} y^{-1} = (H-h)py^{-1}  \geq (H-c_3) py^{-1} \, . $$
As $p^2 \leq |A|^2 \leq c_{12} H^2$, and using Lemma \ref{Lemma_QonPBdd}
$$  y^{-2} = p^2 + q^2 \leq (1 + c_{15}^2) p^2 \, ,$$ 
$$ \frac{d}{dt} y^{-1} \geq \left( \frac{1}{\sqrt{c_{12}}}p-c_3 \right)\frac{1}{\sqrt{(1 + c_{15}^2)}}y^{-2}  \, . $$
As $S$ contains the singularity, for $t$ near $T$, we have $p \to \infty$. For  $p \geq 2\sqrt{c_{12}}c_3$ 
\begin{equation*}
\frac{1}{\sqrt{c_{12}}}p - c_3 \geq \frac{1}{2\sqrt{c_{12}}}p \, . 
\end{equation*}
\noindent
Therefore
$$ \frac{d}{dt} y^{-1} \geq \frac{1}{2\sqrt{c_{12}}}p \frac{1}{\sqrt{(1 + c_{15}^2)}}y^{-2} \geq \frac{1}{2\sqrt{c_{12}}(1 + c_{15}^2)} y^{-3} \, .$$
 Let $U(t) = \max_{S_t} y^{-1} \, . $  By renaming the constant $\frac{1}{2\sqrt{c_{12}}(1 + c_{15}^2)} = \epsilon$ we obtain
 $$ \frac{d}{dt} U(t) \geq \epsilon U^3 (t) \Leftrightarrow \frac{d}{dt} U^{-2}(t) \leq -2\epsilon \, . $$
 Since $U^{-2}(t)$ tends to zero as $t \rightarrow T$, we integrate from $t$ to $T$ and obtain
 \begin{equation*}
 U(t) = \max_{S_t} y^{-1} \leq \frac{1}{\sqrt{2\epsilon(T-t)}} \, .
 \end{equation*}
 As $|A|^2 \leq c_{12} H^2$ and $H = k + p \leq c_1p + p \leq (c_1+1) y^{-1} $ we get the result.
 \end{proof}

 These results will be useful when we consider the different cases outlined below. 

\subsection{Different cases}
\noindent
Due to the lower height bound in $\breve{\Omega}$ ( Lemma \ref{Lemma_yLowerBound}), and Proposition \ref{Prop_AsquaredBddAwayFromAxis}, we know that $|A|^2$  is bounded in $\breve{\Omega} $. Therefore the singularity can only develop in $\hat{\Omega}$.  Also, from Lemma \ref{Lemma_yLowerBound} we know that $ \inf_{ \Gamma_{ \hat{\Omega}}} y \geq c' > 0 $. Furthermore, from Lemma  \ref{Lemma_width_of_hat_Omega_dashed} the projection of $\hat{\Omega}$ onto the $x_1$ axis   $\mathscr{H}^1(I(\hat{\Omega})) > c > 0$.  Keeping these results in mind we  consider two scenarios, depending on the value of $|A|^2/H^2$ in $\hat{\Omega}$: {\bf Case I} $\frac{|A|^2}{H^2} \leq c $ for $t < T$ for some $c$; {\bf Case II} $\frac{|A|^2}{H^2}$ is unbounded in $\hat{\Omega}$ .   \\

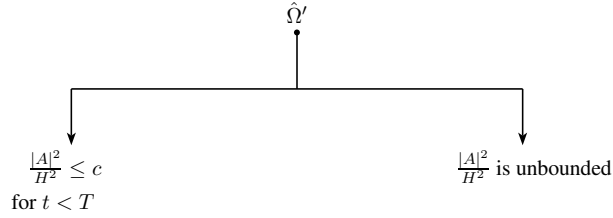
\begin{figure}[h!]
   \centering
   \scalebox{0.75}{\input{differentCases}}\label{fig:differentCases}
\end{figure}

\subsection{Case I: $\frac{|A|^2}{H^2} \leq c $ for $t < T$ in $\hat{\Omega}$}\label{Section_Chap_7_Case_1}
\noindent

From Lemma \ref{Lemma_yLowerBound} we know that $ \inf_{ \Gamma_{ \hat\Omega}}  y  \geq c' > 0\, .$ By letting $S_t := \hat{\Omega}_t\, ,$ from Lemma \ref{Lemma_QonPBdd} and Proposition \ref{Prop_TypeI} we conclude that the singularity is of type I.

\subsection{Case II: $\frac{|A|^2}{H^2}$ is unbounded in $\hat{\Omega}$}\label{Section_Chap_7_Case_2}
\noindent

For this case we prove that a singularity that develops in $\hat{\Omega}$  is of type I by way of contradiction by  using a rescaling procedure similar to that used in \cite{HS2}. A similar rescaling argument is used in \cite{JHSK1} to prove that no singularities can develop if the mean curvature of the surface is bounded in $[0,T)$. \\

\begin{proposition}\label{PropStatementFalse} If $\frac{|A|^2}{H^2}$ is unbounded in $\hat{\Omega}$, then a singularity that develops in $\hat{\Omega}$ is of type I.
\end{proposition}
\begin{proof}
Let us assume a singularity that develops in  $\hat{\Omega}$ is of type II. In order to understand the singularity better, we use a rescaling procedure similar to that in \cite{HS2}. We choose a sequence $\left(l_i, t_i \right)$ as follows. For any integer $i \geq 1$, \, let $t_i \in \left[0, T - 1/i \right]$, \, $l_i \in M^2$, such that
$$ |A|^2(l_i, t_i)\left(T - \frac{1}{i} - t_i \right) = \max_{\substack{ l \in M^{2} \\ t \leq T-\frac{1}{i}} } |A|^2(l, t)\left(T - \frac{1}{i} - t \right) \,. $$

\noindent
Let
\begin{equation} \label{Eq_New_Rescaling_1}
 \alpha_i = |A|(l_i, t_i) \,,  \hspace{5mm} c_i = \alpha_i^2 \left(T - \frac{1}{i} - t_i \right)\, , \hspace{3mm} \text{and} \hspace{3mm} \chi_i = -\alpha_i^2 t_i \, . 
\end{equation}

\noindent
Similar to Lemma $4.3$ in \cite{HS2}, if the singularity is of type II, as $i \to \infty$  we have $t_i \to T$, $\alpha_i \to \infty$, $c_i \to \infty$ and $\chi_i \to -\infty$.

\noindent
From \eqref{Eq_New_Rescaling_1} we obtain 
\begin{equation}\label{eq_alpha_i_proof_contradiction}
 \alpha_i =  |A|(l_i, t_i) \geq \sqrt{\frac{c_i}{T-t_i}} \, .
\end{equation}

We consider the family of rescaled surfaces $\mathscr{M}_{i, \tau}$ defined by the following immersions:
\begin{equation} \label{TypeIHbddEq_1}
 \tilde{\mathbf{x}}_i(\cdot, \tau) = \alpha_i \left( \mathbf{x}(\cdot, \alpha_i^{-2} \tau + t_i )	- \langle \mathbf{x}(l_i, t_i) , \mathbf{i}_1 \rangle \mathbf{i}_1 \right)  \, , 
\end{equation} 
where $ \tau \in [ \chi_i, 0] \, .$	 This rescaling is different from the type II rescaling used in \cite{HS2} due to the following reasons: it is rescaled from a point on the axis of rotation; and the rescaled time interval is different, i.e. for every $i$, we rescale the original surface for $t \in [0, t_i]$.

\noindent
For this rescaling we have
$$  \tilde{H_i}( \cdot, \tau ) = \alpha_i^{-1} H(\cdot, \alpha_i^{-2} \tau + t_i ) \,  \hspace{3mm} \text{and} \hspace{3mm} |\tilde{A}_i|( \cdot, \tau ) = \alpha_i^{-1}	|A| (\cdot, \alpha_i^{-2} \tau + t_i )  \, .$$

This rescaling guarantees that $|\tilde{A}_i| \leq 1$ for $t \leq t_i$. From Proposition \ref{Prop_KonP2} we know that $\frac{|k|}{p} \leq \max(c_1, 1)$ in $\hat{\Omega}$. Therefore 
$$  |A| = \sqrt{k^2 + p^2} \leq c_5p \leq c_5y^{-1} \, , $$
where $c_5= \sqrt{1+ (\max(c_1,1))^2 }$. Hence we obtain
\begin{equation}\label{HbddEq_4}
 |\tilde{A}_i| = \alpha_i ^{-1} |A| \leq \alpha_i ^{-1} c_5 y^{-1} = c_5 (\alpha_i y)^{-1} = c_5 \tilde{y}^{-1} \, . 
\end{equation}
\noindent
Thus the rescaled surfaces do not float away to infinity. 
\noindent
We note that the rescaled surfaces $M_{i, \tau}$ defined by \eqref{TypeIHbddEq_1} also evolve by volume-preserving mean curvature flow. This can be shown by observing that $d\tilde{\mu}_\tau = \alpha_i^2d{\mu}_t $ and $\tilde{h}_i(\tau) = \alpha_i^{-1}h(t)$, so that
\begin{equation}\label{eq_rescaled_vol_flow}
\frac{d}{d\tau} \tilde{x}_i = - \left(\tilde{H}_i - \tilde{h}_i\right)\nu \,.
\end{equation}

The uniform curvature bound $|\tilde{A}_i| \leq 1$ gives rise to uniform bounds on all covariant derivatives of the second fundamental form, see for example \cite{GH3}. By  a standard method, based on the Arzela-Ascoli theorem, we can therefore find a subsequence which converges uniformly in $C^{\infty}$ on compact subsets of $\mathbb{R}^3 \times \mathbb{R}$ to a non-empty smooth limit flow which exists on the interval $ \tau \in (-\infty, 0)$.  \\

In order to analyse the obtained limit flow, which we label by $\tilde{M}_{\infty, \tau}$, we will next show that the sequence $\{ \tilde{H}_i \}$ converges to zero along different paths in $\hat{\Omega} $ approaching the singularity. As $ H  > \frac{c_2}{2}$ in $\hat{\Omega}$, $\frac{|A|^2}{H^2}$ can only be unbounded near the singularity. That is, there exists a neighbourhood  $N_\epsilon$ around the singularity $\left( \mathbf{x}_*, T \right) \in \hat{\Omega}$ such that  $\frac{|A|^2}{H^2} > \frac{1}{\epsilon}$ for any small $\epsilon >0 $. In the neighbourhood $N_\epsilon$, we have $\frac{H}{|A|}  < \epsilon $. As 
$$ \tilde{H} = \alpha_i^{-1} H \leq \frac{H}{|A|} <  \epsilon \, , $$
\noindent
we have $ \tilde{H} \to 0 $  near the singularity. Outside $N_\epsilon$, but in $\hat{\Omega}$,  $H$ is bounded making $ \tilde{H} = \alpha_i^{-1} H \to 0 $ as $\alpha_i \to \infty$ . We conclude that on all paths in $\hat{\Omega} \, , \tilde{H}_i$ converges to zero as $i$ goes to infinity. \\  


The limiting solution $\mathscr{M}_{\infty, \tau}$ is a catenoid, which we rename by $\hat{\mathscr{M}}$,  as it is the only axially symmetric minimal surface with zero mean curvature. \\ \\

We are now in a position to show that we have a contradiction: In order to get a better understanding of the original surface we rescale back $\mathscr{M}_{i ,\tau}$ for large $i$, and show that the estimate $vy \leq c_4$ would not hold on that (the original) surface. \\

\noindent
We denote the quantities associated to the catenoid $\hat{\mathscr{M}}$ by a hat $\hat{\hspace{1mm} }$.  We obtain the catenoid $\hat{\mathscr{M}}$ by rotating $\hat{y} = c_5 \cosh(c_5^{-1}\hat{x}_1) \,  $ around the $x_1$ axis, where $\hat{x}_1$ is the $x_1$ coordinate of the limiting surface $\hat{\mathscr{M}}$. For any $\epsilon_1>0$ and for any $l_0 \in M^n$  we have (since $\mathscr{M}_{i ,\tau}$ converge to $\hat{\mathscr{M}}$)

$$  |\hat{v}(l_0)\hat{y}(l_0) - \tilde{v}_i (l_0, \tau)\tilde{y}_i(l_0, \tau)| \leq \epsilon_1  \hspace{5mm} \text{for large} \hspace{2mm} i \, .$$

\noindent
For the catenoid $\hat{v} = \sqrt{1 + \hat{y}'^2}=\sqrt{1 + \sinh^2(c_5^{-1}\hat{x}_1)} = \cosh(c_5^{-1}\hat{x}_1)$. As $\tilde{y}_i = \alpha_i y $ and $ \tilde{v}_i =  v $ we have

$$   c_5 \cosh^2(c_5^{-1}\hat{x}_1(l_0))  -\epsilon_1 \leq  \alpha_i v(l_0\, , \alpha_i^{-2}\tau + t_i) y (l_0\, , \alpha_i^{-2}\tau + t_i) \, , $$

\begin{equation}\label{eq_Hbdd_eq6}
 \frac{c_5}{ 2\alpha_i}\left( \cosh (2c_5^{-1}\hat{x}_1) +1  \right) - \frac{\epsilon_1}{ \alpha_i} \leq vy  \hspace{5mm} \text{for} \hspace{2mm} i >I_0 \, .
\end{equation}

\noindent
For a given $i$, we can find values of $\cosh (2c_5^{-1}\hat{x})$ as large as we want. Therefore, for a given $i$ we can find many points $\hat{x}_1$ such that
\begin{equation}\label{neweq_Hbdd_neweq1}
 \frac{c_5}{ 2\alpha_i}\left( \cosh (2c_5^{-1}\hat{x}_1) +1  \right) - \frac{\epsilon_1}{ \alpha_i} >> c_4
\end{equation}


\noindent  
From Lemma \ref{Lemma_UV} we know that $vy \leq c_4$. Therefore \eqref{neweq_Hbdd_neweq1} and \eqref{eq_Hbdd_eq6}
contradict Lemma \ref{Lemma_UV}: by examining the rescaled surfaces we find that the estimate $vy \leq c_4$ does not hold on the corresponding, non-rescaled, hypersurfaces near the singular time $T$. 

\noindent
 Therefore we have a contradiction to the original assumption that the singularity is of type II.  
\end{proof}

\noindent
Hence there exists a constant $c >0$ such that for all  $t \in [0, T)\, ,$
$$ \max_{l \in M^2} |A|^2(l, t) \leq \frac{c}{T-t} \, .
$$
\end{proof}

The combination of cases I and II, gives the proof of Theorem \ref{Theorem_TypeI}.

\section*{Appendix : Proof of the non-cylindrical maximum principle }\label{sec:Appendix}
\noindent
{\bf Proof of Proposition \ref{Prop_Max_Principles} } \\
{\bf Part A. } We show that 
\begin{equation}\label{eq:Max_Prin4}
\sup_V f \leq  \sup_{\Gamma_V} f  \, .
\end{equation}
Let $\tilde{f} = f - \epsilon_1 t \, ,$ where $\epsilon_1 > 0\, .$ It holds $\tilde{f}(l, 0) =  f(l, 0)$ and on $\Gamma_V$, $\tilde{f}(l, t) \leq f(l, t)$. We note that 
$$ \frac{d \tilde{f}}{d t} =  \frac{d f}{d t} - \epsilon_1 \, , \hspace{3mm} 
\Delta \tilde{f} = \Delta f \, , \hspace{3 mm} \nabla \tilde{f} = \nabla f \, .
$$ 
Therefore 
$$ \left( \frac{d}{d t} - \Delta  - a\cdot \nabla \right) \tilde{f} < 0 \, .$$
\noindent
At any interior maximum of $\tilde{f}$ , the standard derivative criteria for a the local maximum say
$$ \frac{d \tilde{f}}{d t} \geq 0 \, , \hspace{3mm} \frac{\partial \tilde{f}}{\partial x_i} = 0 \, , \hspace{3mm} \frac{\partial ^2 \tilde{f}}{\partial x_i \partial x_j} \leq 0 \, .
$$
\noindent
As 
\begin{equation}\label{eq:Max_Prin1}
\Delta \tilde{f} =g^{ij}\left( \frac{\partial^2 \tilde{f}}{\partial x_i\partial x_j} - \Gamma^k_{ij}\frac{\partial \tilde{f}}{\partial x_k} \right) \,  \hspace{3mm} \text{and} \hspace{3mm} \nabla \tilde{f}  = g^{ij}\frac{\partial \tilde{f} }{\partial x_j}\frac{\partial \mathbf{x}}{\partial x_i} \, , \hspace{3mm}
\end{equation}

\noindent
by choosing normal coordinates, such that $g_{ij} = \delta_{ij}$ at the point that corresponds to the interior maximum, we have
$$ \left( \frac{d}{d t} - \Delta  - a\cdot \nabla \right) \tilde{f} \geq 0 \, . $$
\noindent
This is a contradiction. Hence $\tilde{f}(l, t)$ is bounded by the values of $\sup_{\Gamma _V} \tilde{f} $ at all times. Therefore
$$ \sup_V \tilde{f} \leq \sup_{\Gamma _V} \tilde{f} \leq  \sup_{\Gamma_V} f  \, , $$
$$ \sup_V  f(l, t) - \epsilon_1 T \leq \sup_V \left( f(l, t) - \epsilon_1 t \right) \leq \sup_{\Gamma_V} f(l, t) \, , $$
$$  \sup_V f(l, t) \leq \sup_{\Gamma_V} f(l, t) + \epsilon_1 T \hspace{3mm} \text{for all } \hspace{3mm} \epsilon_1 > 0 \, , $$
\noindent
giving us
$$  \sup_V f(l, t) \leq \sup_{\Gamma_V} f(l, t) \, , $$
\noindent
which completes {\bf Part A} . \\
\noindent
{\bf Part B.} For this part we suppose that $f$ has a positive maximum in $\overline{V}$. We will prove by contradiction that
\begin{equation}\label{eq:Max_Prin2}
\sup_V f \leq \sup_{\delta_V} f \, .
\end{equation}
\noindent
Suppose that \eqref{eq:Max_Prin2} does not hold. As $ \sup_V f(l, t) \leq \sup_{\Gamma_V} f $, the maximum of $f$ can only be achieved at an interior point of $Z_V$ for \eqref{eq:Max_Prin2} to be contradicted.  We denote by $Z_{\max}$ the union of $Z_t$'s on which the maximum is achieved. Let  $t_*$ denote the first time that the maximum is achieved on $Z_V$.  Let 
$$ K = \{ (l, t) \in \overline{V} : f(l, t) = \sup_V f  \}  \, . 
$$

\begin{figure}[h]
   \centering
   \scalebox{0.75}{\input{nonCylindrical2}}
		\caption{If the maximum is achieved on $Z_V$}
		\label{fig_nonCylindrical2}
\end{figure}
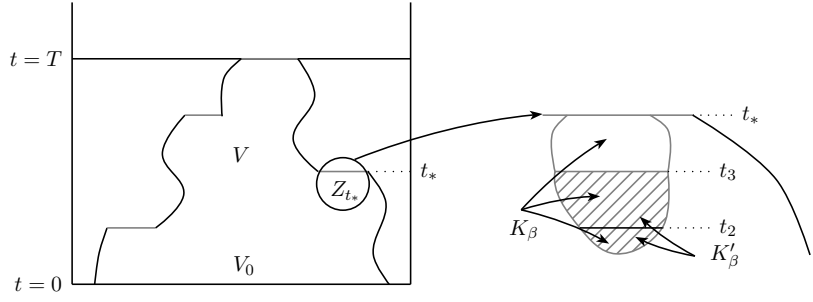

\noindent
(We note that $K \cap \partial V $ is non-empty as the maximum is achieved in $Z_{\max}$ and also that $K \cap  \partial V  \subset Z_V$.) Therefore there exists a $\beta>0$  such that
$$ \sup_V f  > \beta > \sup_{\Gamma_V \backslash Z_{\max}} f \, . $$
\noindent
As $\delta_V \subset \left(\Gamma_V \backslash Z_{\max} \right)$ we have
$$ \sup_{\Gamma_V \backslash Z_{\max}} f \geq  \sup_{\delta_V} f \, . $$
\indent Define
$$ K_\beta = \{ (l, t) \in \overline{V} : f(l,t) \geq \beta  \}  \, . $$
\noindent
We note that $K_\beta$ is not empty and it may not be a connected set. It holds that $Z_{t_*} \subset K_\beta$. We work with the connected component of $K_\beta$, which has $Z_{t_*}$ as a part of its boundary.   Let $ 0 < \epsilon_2 < \sup_{V} f - \beta$. As $V$ is open there exists $(l_2, t_2) \in V$ (depending on $\epsilon_2$) such that $f(l_2, t_2) \geq \sup_V f - \epsilon_2 $. We take $t_2$ to be the earliest time such that $f(l, t) \geq \sup_V f - \epsilon_2 $ is satisfied.  Take $t_3 \in (t_2, t_*)$ and choose a smooth function $\phi:[0, T] \rightarrow \bigR\, ,$ such that $0 \leq \phi \leq 1\, , \, \phi' < 0 $ on $(t_2, t_3)\, ,$ and $ \phi = 1$ on $[0, t_2]$, and  $\phi = 0$ on $[t_3, T].$ The set
$$ K'_\beta =  \{ (l, t) \in K_\beta : t \leq t_3 \} \neq \emptyset \, , $$
\noindent
is a compact subset in $V$. The set $K'_\beta$ is in the interior of $V$, thus $K'_\beta \cap Z_V = \emptyset\, .$  As $(\phi f )(l_2, t_2) = f(l_2, t_2)\, , $ and as the maximum of $f$ is achieved in $Z_{t_*}$, and $\phi \geq 0 \, ,$ the supremum of $\phi f $ in $K_\beta$ must be at least as big as $f(l_2, t_2)$.   Also 
$$ \max_{K_\beta} \phi f = \max_{K'_\beta} \phi f \, , $$
\noindent
as $\phi = 0$ on $[t_3, T].$ Hence
$$\max_{K'_\beta} \phi f \geq f(l_2, t_2) \geq \sup_V f - \epsilon_2 > \beta > 0 \, . $$

On the other hand, $V\backslash K'_\beta = (V\backslash K_\beta) \cup \left( K_\beta \backslash K'_\beta \right)$. On $(V \backslash K_\beta) \cap \{(l, t) \in V : f \geq 0 \}$ we have $\phi f \leq f < \beta $, and on $(V \backslash K_\beta) \cap \{(l, t) \in V : f < 0 \}$ we have $\phi f < \beta \, .$ On $K_\beta \backslash K'_\beta\, , \phi = 0 $. As a result
\begin{equation}\label{eq:Max_Princ_3}
 \sup_{V\backslash K'_{\beta}} \phi f < \beta  < \max_{K'_{\beta}} \phi f \, .
\end{equation} 
\noindent
Therefore, if we denote by $(l_\phi, t_\phi)$ the point at which $(\phi f) (l_\phi, t_\phi) = \sup_V (\phi f)\, ,$ we can see that $(l_\phi, t_\phi) \in K'_\beta \, .$  As
$$ \frac{d }{d t}(\phi f) = \phi \frac{d  f}{d t} + \phi' f  \, , \hspace{3mm} 
\Delta \left(\phi f \right) = \phi\Delta  f \, , \hspace{3 mm} \nabla \left(\phi f\right) = \phi\nabla  f \, , $$
\noindent
we have 
$$\left( \frac{d}{d t} - \Delta  - a\cdot \nabla \right)(\phi  f) = \phi\left( \frac{d}{d t} - \Delta  - a\cdot \nabla \right)  f + \phi' f \, . $$
\noindent
As $\phi'(l_\phi, t_\phi) \leq 0 $ and $ f(l_\phi, t_\phi) > 0 $ we obtain
$$ \left( \frac{d}{d t} - \Delta  - a\cdot \nabla \right) (\phi f) \leq 0 \, .$$
\noindent
By using Part A with $\phi f$ replacing $f$ we have 
$$ \sup_V (\phi  f) \leq \sup_{\Gamma_V} (\phi  f) \, . $$
But this is a contradiction as $\sup_V (\phi  f) = (\phi  f) (l_\phi, t_\phi)$ with $(l_\phi, t_\phi) \notin \Gamma_V\, ,$ and because $\Gamma_V \subset \overline{V} \backslash K'_\beta$ and \eqref{eq:Max_Princ_3} holds. Therefore our original assumption is wrong. Hence a maximum of $ f$ does not occur in $Z_V$, that means $K \cap \partial V \not\subset Z_V\, .$ Therefore we conclude that \eqref{eq:Max_Prin2} is true. \\
\noindent
This concludes the proof of Proposition \ref{Prop_Max_Principles}.

\bibliographystyle{acm}
\bibliography{citations}
\end{document}

%% file: nonCylindrical1.tex
\pagestyle{empty}
\newrgbcolor{xdxdff}{0.49 0.49 1}
\psset{xunit=1.0cm,yunit=1.0cm,dotstyle=o,dotsize=3pt 0,linewidth=0.8pt,arrowsize=3pt 2,arrowinset=0.25}
\begin{pspicture*}(-4.3,-0.76)(3.86,5.64)
\psline(-3,0)(-3,5)
\psline(3,0)(3,5)
\psline(-3,0)(3,0)
\psline(-3,4)(3,4)
\pscurve(-2.6,0)(-2.52,0.56)(-2.38,1)
\psline[linecolor=gray](-2.38,1)(-1.52,1)
\pscurve(-1.52,1)(-1.02, 1.6)(-1.4, 2.32)(-1,3)
\psline[linecolor=gray](-1,3)(-0.34,3)
\pscurve(-0.34,2.98)(-0.3, 3.68)(0,4)
\pscurve(1,4)(1.3, 3.36)(0.92, 2.56)(1.36,2)
\psline[linecolor=gray](1.36,2)(2.24,2)
\pscurve(2.24,2)(2.56, 1.46)(2.34, 0.52)(2.62,0)
\psline[linestyle=dotted](-2.38,1)(-3,1)
\psline[linestyle=dotted](-1,3)(-3,3)
\psline[linestyle=dotted](2.24,2)(3,2)
\uput[180](-3,0){$ t= 0 $}
\uput[180](-3,1){$t_1$}
\uput[180](-3,3){$t_3$}
\uput[180](-3,4){$t = T$}
\uput[0](3, 2){$t_2$}
\uput[90](0,0){ $V_0$}
\uput[90](-2,1){\gray  $Z_{t_1}$}
\uput[90](-0.68,3){\gray  $Z_{t_3}$}
\uput[90](1.8,2){\gray  $Z_{t_2}$}
\uput[90](0, 2){$V$}
\end{pspicture*}

%% file: NegativeHBdd.tex
\pagestyle{empty}
\psset{xunit=1.0cm,yunit=1.0cm,dotstyle=o,dotsize=3pt 0,linewidth=0.8pt,arrowsize=3pt 2,arrowinset=0.25}
\begin{pspicture*}(-3.28,-0.84)(6.58,6.3)
\psline(-2,5)(-2,0)
\psline(5,5)(5,0)
\psline(-2,4)(5,4)
\psline(-2,0)(5,0)
\psarc[showpoints=true](2,2.5){0.4}{180}{360}
\uput[180](-2,0){$t =0$}
\uput[180](-2,2){$t$}
\uput[180](-2,4){$t =T$}
\uput[-90](1,0){$\Omega_0$}
\uput[60](2,2.5){\tiny $\mathbf{x}(l_0, t_0)$}
\psdot*(2,2.5)
\uput[-110](2,2.5){\tiny $N$}
\psarc[showpoints=true](1.75, 4){0.4}{180}{360}
\uput[90](1.75, 4){\tiny $\mathbf{x}(l_1, T)$}
\uput[-110](1.75, 4){\tiny $N$}
\psdot*(1.75, 4)

\pscurve(1, 4) (0, 3)(2,1)(3,3)(2,3)(3,4)
\pscurve(1.75, 4) (0.25, 3)(2,1.75)(2,2.5)(1.75,3)(2.75,4)
\uput[90](2,1){$\breve{\Omega}'$}
\uput[0](0.25, 3){\hspace{4mm}$  \breve{\Omega}$}
\end{pspicture*}

%% file: TheTwoPaths.tex
\pagestyle{empty}
\psset{xunit=1.0cm,yunit=1.0cm,dotstyle=o,dotsize=3pt 0,linewidth=0.8pt,arrowsize=3pt 2,arrowinset=0.25}
\begin{pspicture*}(-4.5,-0.5)(5,5.3)
\psline(-3,5)(-3,0)
\psline(4,0)(4,5)
\psline(-3,4)(4,4)
\psline(-3,0)(4,0)
\uput[180](-3,0){$t =0$}
\uput[180](-3,2){$t$}
\uput[180](-3,4){$t =T$}
\pscustom[fillstyle=hlines,linecolor=brown,hatchcolor=orange]{
\pscurve[linecolor=brown](-2.76, 0)(-2.46, 1)(-2, 2)(-2.2, 3)(-2, 4)
\psline(-2, 4)(-3,4)
\psline(-3,4)(-3,0)
\psline(-3,0)(-2.76, 0)}
\pscustom[fillstyle=hlines,linecolor=brown,hatchcolor=orange]{
\pscurve[linecolor=brown](-1, 0)(-1, 1)(-0.8, 1.7)(-0.74, 2.24)(-0.56, 3)(-0.5, 4)
\psline(-0.5, 4)(1.3, 4)
\pscurve[linecolor=brown](1.3, 4)(0.98, 3.42)(0.84, 2.9)(0.8, 2.24)(0.78, 1.54)(0.44, 1.02)(0.38, 0)
\psline(0.38, 0)(-1, 0)}

\pscustom[fillstyle=hlines,linecolor=brown,hatchcolor=orange]{
\pscurve[linecolor=brown](4, 2)(3.46, 2.46)(3, 3)(2.62, 4)
\psline(2.62, 4)(4,4)
\psline(4,4)(4, 2)}

\pscurve[linecolor=blue](-0.66, 0)(-0.62, 0.76)(-0.34, 1.6)(-0.18, 2.46)(-0.1, 3.38)(0, 4)
\pscurve[linecolor=red](-0.18, 0)(-0.08, 0.64)(0.14, 1.28)(0.1, 1.78)(0.16, 2.3)(0.26, 2.86)(0.38, 3.34)(0.34, 3.74)(0.28, 4) 
\psline{->}(1.48,1.78)(0.1,1.78)
\psline{->}(-1.22,0.76)(-0.62,0.76)
\uput[0](1.48,1.78){$H = C_1 $}
\uput[180](-1.22,0.76){$H = C_2 $}
 \uput[0](-2.85, 2.52){  $\breve{\Omega}'$}
\uput[0](3.2, 3.28){  $\breve{\Omega}'$}
\uput[90](0.66, 3.3){  $\breve{\Omega}'$}
\end{pspicture*}

%% file: differentCases.tex
\pagestyle{empty}
\psset{xunit=1.0cm,yunit=1.0cm,dotstyle=o,dotsize=3pt 0,linewidth=0.8pt,arrowsize=3pt 2,arrowinset=0.25}
\begin{pspicture*}(-4.3,1)(9.78,6.3)
\psline(1,5)(1,4)
\psline(-3,4)(5,4)
\psline{->}(-3,4)(-3,3)
\psline{->}(5,4)(5,3)
\psdots[dotstyle=*,linecolor=black](1,5)
\uput[90](1, 5){$\hat{\Omega}'$}
\uput[0](-3.94, 2.62){$\frac{|A|^2}{H^2} \leq c $  }
\uput[0](-4.24, 2){for $t < T$}
\uput[0](3.64, 2.62){$\frac{|A|^2}{H^2}$ is unbounded}
\end{pspicture*}

%% file: nonCylindrical2.tex
\pagestyle{empty}
\newrgbcolor{xdxdff}{0.49 0.49 1}
\psset{xunit=1.0cm,yunit=1.0cm,dotstyle=o,dotsize=3pt 0,linewidth=0.8pt,arrowsize=3pt 2,arrowinset=0.25}
\begin{pspicture*}(-4.3,-0.76)(11.16,5.64)
\newrgbcolor{xdxdff}{0.49 0.49 1}
\psline(-3,0)(-3,5)
\psline(3,0)(3,5)
\psline(-3,0)(3,0)
\psline(-3,4)(3,4)
\pscurve(-2.6,0)(-2.52,0.56)(-2.38,1)
\psline[linecolor=gray](-2.38,1)(-1.52,1)
\pscurve(-1.52,1)(-1.02, 1.6)(-1.4, 2.32)(-1,3)
\psline[linecolor=gray](-1,3)(-0.34,3)
\pscurve(-0.34,2.98)(-0.3, 3.68)(0,4)
\psline[linecolor=gray](0,4)(1,4)
\pscurve(1,4)(1.3, 3.36)(0.92, 2.56)(1.36,2)
\psline[linecolor=gray](1.36,2)(2.24,2)
\pscurve(2.24,2)(2.56, 1.46)(2.34, 0.52)(2.62,0)
\psline[linestyle=dotted](2.24,2)(3,2)
\uput[180](-3,0){$ t= 0 $}
\uput[180](-3,4){$t = T$}
\uput[0](3, 2){$t_*$}
\uput[90](0,0){ $V_0$}
\uput[-90](1.8,2){  $Z_{t_*}$}
\uput[90](0, 2){$V$}
\pscircle(1.8,1.78){0.48}
\psline[linecolor=gray](5.34, 3)(8, 3)
\pscustom[fillstyle=hlines,linecolor=gray,hatchcolor=gray]{
\psline(7.56, 2)(5.56, 2)
\pscurve(5.56,2)(6, 1)(6.64,0.54)(7.46, 1)(7.56,2)
}
\pscurve[linecolor=gray](5.78, 3)(5.52, 2.7)(5.56,2)
\pscurve[linecolor=gray](7.24,3)(7.56, 2.74)(7.56, 2)
\pscurve(8, 3)(9.48,1.86)(10.08,0.52)
\pnode(2.00728, 2.2134){A}
\pnode(5.34, 3.02){B}
\ncarc{->}{A}{B}

\psline[linestyle=dotted](8.3, 1)(7.46, 1)
\psline[linestyle=dotted](8.3, 2)(7.56, 2)
\psline[linestyle=dotted](8.7, 3)(8,3)
\uput[0](8.3, 2){$ t_3$}
\psline(6, 1)(7.46, 1)
\uput[0](8.3, 1){$ t_2$}
\uput[0](8.7, 3){$ t_*$}
\pnode(6.48, 2.58){C}
\pnode(6.36, 1.56){D}
\pnode(6.52, 0.72){E}
\pnode(4.96, 1.32){F}
\pnode(8.04, 0.5){H}
\pnode(7.06, 1.2){G}
\pnode(6.98, 0.84){I}
\ncarc{->}{F}{C}
\ncarc{->}{F}{D}
\ncarc{->}{F}{E}
\ncarc{->}{H}{G}
\ncarc{->}{H}{I}
\uput[0](8.04, 0.5){ $K'_\beta$}
\uput[-90](4.96, 1.32){ $K_\beta$}
\end{pspicture*}